\documentclass[12pt,reqno]{amsart}

\usepackage{graphicx}

\usepackage{subcaption} 
\usepackage{adjustbox}
\usepackage{tikz-cd}
\usepackage{hyperref}
\hypersetup{
	unicode=true,
	colorlinks=true,
	citecolor=black,
	linkcolor=black,
	urlcolor=black,
	plainpages=false,
 }

 \usepackage{url}	
 \allowdisplaybreaks 

\usepackage{pgf}

\usepackage{xcolor}

\usepackage{comment} 

\newtheorem{teo}{Theorem}[section]

\newtheorem{corollary}[teo]{Corollary}

\newtheorem{lemma}[teo]{Lemma}
\newtheorem{prop}[teo]{Proposition}

\theoremstyle{definition}
\newtheorem{defi}[teo]{Definition}
\newtheorem{example}[teo]{Example}

\theoremstyle{remark}
\newtheorem{rem}[teo]{Remark}

\numberwithin{figure}{section}

\newcommand{\Ba}{\mathcal{B}}

\newcommand{\bbar}{\overline}

\newcommand{\CC}{\mathcal{C}}

\newcommand{\Cart}{\mathrm{Cart}}

\newcommand{\co}{\colon}

\newcommand{\D}{\mathcal{D}}

\newcommand{\Dc}{\mathcal{D}}

\newcommand{\Si}{\mathcal{S}}

\newcommand{\Ho}{\mathrm{H}}

\newcommand{\E}{\mathcal{E}}

\newcommand{\F}{\mathcal{F}}

\newcommand{\id}{\mathrm{id}}

\newcommand{\N}{\mathrm{N}}

\newcommand{\Ima}{\mathrm{Im}}

\newcommand{\sC}{\mathrm{sc}}

\newcommand{\sG}{\mathrm{Sg}}

\newcommand{\op}{\mathrm{op}}

\newcommand{\opCart}{\mathrm{opCart}}

\newcommand{\Ab}{\textbf{\textsf{Ab}}}

\newcommand{\Uc}{\mathcal{U}}

\newcommand{\Vc}{\mathcal{V}}

\newcommand{\wtilde}{\widetilde}

\newcommand{\cpl}{\mathrm{c.p.l}}

\begin{document}

\title[Cohomology and \v{S}varc genus]{Baues-Wirsching cohomology and\\ \v{S}varc genus in small categories}


\author[I. Carcac\'ia]{%
	Isaac Carcac\'ia-Campos 
}

 \address{%
           	Isaac Carcac\'ia-Campos
            \\
              CITMAga,Departamento de Matem\'aticas, Universidade de Santiago de Compostela, 15782-SPAIN}
               \email{isaac.c.campos@usc.es}

\author[E. Mac\'ias]{%
	Enrique Mac\'ias-Virg\'os 
}

 \address{%
           	Enrique Mac\'ias-Virg\'os \\
              CITMAga, Departamento de Matem\'aticas, Universidade de Santiago de Compostela, 15782-SPAIN}
               \email{quique.macias@usc.es}

\author[D. Mosquera]{%
	David Mosquera-Lois 
}

 \address{%
           	David Mosquera-Lois \\
              Departamento de Matemáticas, Universidade de Vigo, 15782-SPAIN}
               \email{david.mosquera.lois@uvigo.gal}


\begin{abstract} We prove that for a bifibration 
$P$ between small categories, the length of the cup product in the kernel of the induced morphism ${P^*}$ in the Baues-Wirsching cohomology with coefficients in any natural system is a lower bound for the
\v{S}varc genus of $P$. Our results extend classical \v{S}varc-type inequalities to the categorical setting and introduce a computationally efficient method via a reduced cochain complex for Baues–Wirsching cohomology.
\end{abstract}

\keywords{Small category, Fibration, Baues-Wirsching Cohomology, \v{S}varc genus, sectional category, sectional number, cup length}



\thanks{We extend our gratitude to J. M. García-Calcines and P. Pavešić for bringing to our attention the existence of Baues-Wirsching cohomology. We also thank F. Muro for insightful discussions on the topic.}

\maketitle


\section*{Introduction}\label{sec: Prelim}

The \v{S}varc genus of a continuous map $f\colon X\to Y$ is a well-known homotopy invariant which, in an informal sense, measures how far $Y$ is from being dominated by $X$ through $f$. Formally, the {\em \v{S}varc genus} of $f$ is the least integer $n\geq 0$ such that there is a cover $\{U_0, \ldots, U_n \}$ of $Y$ by open subspaces that admit local homotopic sections. Specifically, for every $i \in \{0, \ldots, n\} $ there is a continuous map $s_i \colon U_i \rightarrow X$ that satisfies $f \circ s_i \simeq \iota_i$ where $\iota_i \colon U_i \rightarrow Y$ is the inclusion (see \cite{Svarc}). If we impose that the maps $s_i \colon U_i \rightarrow X$ are true sections (not only up to homotopy), then we define the sectional number of the map $f$. Both notions are equivalent for fibrations. The \v{S}varc genus of a map is not only significant in itself but also generalizes two other important notions: namely the LS-category (\cite{CLOT}) and the topological complexity (\cite{FARBER}). 

The study of topological invariants of small and acyclic categories has recently gained importance in different areas of research such as Computational Topology, Geometry, Combinatorics and simplicial methods. This importance stems, among other factors, from the suitability of such structures for developing algebraic combinatorial models that reflect interactions between different objects offering complexity and flexibility that are not achievable by simplicial complexes or preorders (see, for example, \cite{Bridson_Haefliger,KOZLOV}). 

The LS-category and the topological
complexity were studied in the context of small categories in \cite{TANAKA} and \cite{MAC-MOSQ-FUNCT}. In this work, we undertake the task of defining and studying the notions of sectional number and \v{S}varc genus in the context of small categories. 

In the context of topological spaces, the length of the cup product of the kernel of the induced morphism $p^*\colon \Ho(B)\to \Ho(E)$ in cohomology serves as a lower bound for the Švarc genus of a fibration $p\colon E\to B$ (see \cite{Svarc}). This inequality provides excellent lower bounds for computing the LS-category and the topological complexity of spaces. 

In the context of small categories, Baues and Wirsching introduced a cohomology theory with coefficients in a natural system $D$ endowed with an endopairing (\cite{BAUES0, BAUES1}). Our main result in this work is the following Švarc type inequality in the context of small categories for Baues-Wirsching cohomology: the length of the cup product in $\ker\Ho(P;D)$ is a lower bound for the \v{S}varc genus of $P$. Moreover, we develop a computational improvement for such cohomology by means of a \emph{reduced} cochain complex which greatly simplifies and facilitates the computations.

The paper is organized in two sections. In Section \ref{sec: COH} we further develop the notion of relative Baues-Wirsching cohomology \cite{Baues1989TheCO} and introduce the novel notion of relative cup product in order to define cup length in this context. Furthermore, we introduce an improvement in the computation of Baues-Wirsching cohomology by introducing a \emph{reduced} cochain complex. To improve readability and ensure the exposition is self-contained, we have included 
some preliminaries on Baues-Wirsching cohomology from scratch and a fully detailed example (Subsection \ref{subsec:example}).

Section \ref{sec:bifibrations_Svarc} deals with the following topic.  First, we prove some properties of bifibrations and we highlight a remarkable family of bifibrations: coverings of small categories. Second, we introduce the original notions of sectional number and  \v{S}varc genus of a functor and prove that both notions coincide for bifibrations (mirroring what happens in the topological context). Finally, we prove the main result of this work: Theorem \ref{thm:main_thm}. This result is a \v{S}varc inequality for bifibrations which claims that the length of the cup product of the bifibration (a purely algebraic invariant) is a lower bound for its \v{S}varc genus.

In this paper we will assume that all categories are small unless stated otherwise. Moreover, if $\CC$ is a category, we will denote by $\mathrm{Obj}(\CC)$ its set of objects, by $\mathrm{Mor}(\CC)$ the morphisms of $\CC$ and by $\CC(c_1,c_")$ the set of arrows between the objects $c_1,c_2$ in $\CC$.

\subsection*{Funding}The authors were partially supported by  Xunta de Galicia ED431C 2023/31 with FEDER funds and by MICINN-Spain research project PID2020-114474GB-I00.

\section{Cohomology of small categories} \label{sec: COH}

We introduce the notions of relative cohomology  with coefficients in a natural system and relative cup product for small categories (see \cite{Baues1989TheCO}). To this end, we first review the fundamentals of cohomology of small categories with coefficients in a natural system (see \cite{BAUES0,BAUES1}) and establish the notation.  

We encourage parallel reading of the fully worked out example (Subsection \ref{subsec:example}) along with this section. 

\subsection{The factorization category of a category}

Let $\CC$ be a small category. The category of factorizations of $\CC$, denoted by $\F \CC$, is the category whose objects are the morphisms of $\CC$ and whose morphisms from $\lambda \colon A \rightarrow B$ to $\mu \colon C \rightarrow D$ are the factorizations $(\alpha, \beta) \colon \lambda \rightarrow \mu$:
$$
\begin{tikzcd}
B \arrow[r, "\alpha"] & D                                   \\
A \arrow[u, "\lambda"]      & C \arrow[u, "\mu"] \arrow[l, "\beta"]
\end{tikzcd}$$
which make the diagram commutative, that is, $\mu=\alpha\circ \lambda \circ \beta$.
The composition of morphisms is defined as: $$(\alpha', \beta') \circ (\alpha, \beta)=(\alpha' \circ \alpha, \beta \circ \beta').$$



\begin{rem}\label{factorization}
Every morphism $(\alpha, \beta)$ in $\F \CC$ can be  factorized as $(\alpha, \beta)= (\id_D, \beta)\circ (\alpha,\id_A)$, as in the following diagram:
$$
\begin{tikzcd}
B \arrow[r, "\alpha"] & D \arrow[r, "\id_D"]                              & D                                   \\
A \arrow[u, "\lambda"]      & A \arrow[u, "\alpha \circ \lambda"'] \arrow[l, "\id_A"] & C \arrow[u, "\mu"] \arrow[l, "\beta"]
\end{tikzcd}
$$
\end{rem}

\begin{rem}\label{IDENT}
For any category $\CC$ a morphism $\alpha\colon A \rightarrow B$ induces two morphisms $(\alpha,\id_A)\colon\id_A \rightarrow \alpha$ and $(\id_B,\alpha)\colon\id_B \rightarrow \alpha$ in $\F \CC$, as in the following diagram:
$$
\begin{tikzcd}
A \arrow[r, "\alpha"] & B                                        & B \arrow[r, "\id_B"] & B                                         \\
A \arrow[u, "\id_A"]  & A \arrow[u, "\alpha"] \arrow[l, "\id_A"] & B \arrow[u, "\id_B"] & A \arrow[l, "\alpha"] \arrow[u, "\alpha"]
\end{tikzcd}
$$\end{rem}
\begin{rem}\label{ARROW}
    For every category $\CC$, we have the functor $\pi\colon\F \CC \rightarrow \CC$ defined as $\pi(\lambda)=\mathrm{Codom}(\lambda)$ and $\pi(\alpha,\beta)=\alpha$.
\end{rem}
Any functor $F\colon \CC \rightarrow \Dc$ between two categories induces a functor 
     $\widehat{F}\colon \F \CC \rightarrow \F \Dc$  between their factorization categories, defined as follows. For every morphism $\lambda$ we have  $\widehat{F}(\lambda)=F(\lambda)$ and for every factorization $(\alpha,\beta)$ of $\mu$ by $\lambda$ we have  $\widehat{F}(\alpha,\beta)=(F(\alpha),F(\beta))$, which is a factorization because $$F(\alpha)\circ F(\lambda) \circ F(\beta)=F(\alpha \circ \lambda\circ \beta)=F(\mu).$$
     It is easy to check that  $\widehat{G \circ F}=\widehat{G} \circ \widehat{F}$ for two composable functors  $F \colon \CC \rightarrow \Dc$ and $G \colon \Dc \rightarrow \E$.
     
     Moreover, this construction is natural in the sense that we have the following commutative diagram $$
\begin{tikzcd}
\F \CC \arrow[d, "\widehat{F}"'] \arrow[r] &  \CC \arrow[d, "F"] \\
\F \Dc \arrow[r]                 & \Dc                     
\end{tikzcd}
    $$
    where the horizontal arrows  are the functors defined in Remark \ref{ARROW}.  
 

\subsection{Natural systems and cohomology}\label{subsec:natural_systems_cohomology}
    Natural systems will be the coefficients of Baues-Wirching cohomology of small categories (see \cite{BAUES0}). A {\em natural system} in the category $\CC$ is a functor $D\colon \F \CC \rightarrow \Ab$, from the category of factorizations of $\CC$ to the category of abelian groups. We will denote $D(\alpha)$ by $D_\alpha$. Moreover, if $\alpha=\id_A$ we will denote its image by $D_A$.

\begin{rem}\label{alpha_y _beta*}
    Every morphism $D(\alpha, \beta)$ can be factorized using Remark \ref{factorization}. Indeed 
    $$D(\alpha, \beta)= D((\id_D, \beta) \circ (\alpha,\id_A) )= D(\id_D, \beta) \circ D(\alpha,\id_A).$$ 
    We will denote by $\alpha_*$ and $\beta^*$ the morphisms $D(\alpha,\id_A)$ and $D(\id_A,\beta)$ respectively. 
\end{rem}

\begin{rem}
For any functor $F\colon \CC \rightarrow \Dc$ and any natural system $D\colon  \F \Dc \rightarrow \Ab$ in $\Dc$ we have a natural system $F^*D$ in $\CC$ given by  $F^*D=D \circ \widehat{F}$ (see \cite{BAUES0}). 
\end{rem}

Let $n\geq 1$ a natural number. The {\em chains of length $n$ or $n$-chains of $\CC$} are the sequences of $n$ composable morphisms $\lambda_1, \dots,\lambda_n$ in $\CC$:
$$
\begin{tikzcd}
A_0 & \cdots \arrow[l, "\lambda_1"'] & A_{n} \arrow[l, "\lambda_{n}"'].
\end{tikzcd}$$
We denote by $N_n(\CC)$ the set of $n$-chains. Moreover $N_0(\CC)$ will be the set of identity morphisms of $\CC$, or equivalently, the objects of $\CC$. 

Let  $F^n(\CC;D)$ (sometimes denoted simply by $F^n$) be the Abelian group of $n$-cochains of $\CC$ with coefficients in $D$, i.e., the Abelian group of all maps $$f\colon N_n(\CC) \rightarrow \bigcup_{\alpha \in \mathrm{Mor}(\CC)} D_\alpha$$ such that for every $n$-chain $(\lambda_1, \dots, \lambda_n)$ we have that $f(\lambda_1, \dots, \lambda_n) \in D(\lambda_1 \circ \cdots \circ \lambda_n)$. For $n=0$,  $F_0(\CC;D)$ is the set of maps  $f\colon \mathrm{Obj} (\CC) \to \bigcup_{A \in \mathrm{Obj}(\CC)} D_A$ such that $f(A) \in D_A$.

Consider the {\em coboundary operator} 
$$\delta^{n-1}: F^{n-1}(\CC;D) \rightarrow F^{n}(\CC;D)$$ 
given by the following formula:
    \begin{align*}
\delta^{n-1} f(\lambda_1,\ldots, \lambda_{n})=&\lambda_{1*}f(\lambda_2,\ldots, \lambda_{n})\\&+\sum_{i=1}^{n-1} (-1)^i f(\lambda_1, \ldots, \lambda_i \circ \lambda_{i+1}, \ldots, \lambda_{n})  \\&+ (-1)^n \lambda_{n}^* f(\lambda_1,\ldots, \lambda_{n-1}),
\end{align*}
    where the morphisms $\lambda_{1*}$ and $\lambda_n^*$ are generated by the following diagrams in the sense of Remark \ref{alpha_y _beta*}: 
$$
\begin{tikzcd}
A_1 \arrow[r, "\lambda_1"]                               & A_0                                                                                 &  & A_0 \arrow[r, "\id_{A_0}"]                               & A_0                                                                                  \\
A_{n} \arrow[u, "\lambda_2 \circ \cdots \circ \lambda_{n}"] & A_{n} \arrow[l, "\id_{A_{n}}"] \arrow[u, "\lambda_1 \circ \cdots \circ \lambda_{n}"'] &\quad  & A_{n-1} \arrow[u, "\lambda_1 \circ \cdots \circ \lambda_{n-1}"] & A_{n} \arrow[u, "\lambda_1 \circ \cdots \circ \lambda_{n}"'] \arrow[l, "\lambda_{n}"]
\end{tikzcd}
$$


Moreover, for $n=0$ we have $$\delta^0(\lambda)=\lambda_{*}(A)-\lambda^*(B),$$ where $A=\mathrm{Dom}(\lambda)$, $B=\mathrm{Codom}(\lambda)$ and $\lambda_*$ and $\lambda^*$ are defined as in Remark \ref{alpha_y _beta*}.

\begin{defi}[\cite{BAUES0,BAUES1}]\label{def:cohomo_Baues}
The {\em cohomology of the} category $\CC$ with coefficients in the natural system $D$, denoted by $\Ho^n(\CC;D)$, is the cohomology of the cochain complex $(F^n(\CC;D),\delta)$, that is,
$$ \Ho^n(\CC;D)= \frac{\ker \delta^{n}}{\Ima \, \delta^{n-1}}.$$
\end{defi}

The cohomology with coefficients in a natural system is functorial as expected. Let $F\colon \CC \rightarrow \Dc$ be a functor between two categories. and $D$ a natural system in $\Dc$. There is a morphism $F^*\colon \Ho^n(\Dc; D) \rightarrow \Ho^n(\CC;F^*D)$ for all integers $n\geq 0$ (see \cite{BAUES0,BAUES1}). Moreover:

\begin{prop}

    Let $F\colon \CC \rightarrow \Dc$ and $G\colon \Dc \rightarrow \E$ be two functors between small categories and let $D$ be a natural system in  $\E$.  For each integer $n\geq 0$ we have three morphisms:
    \begin{enumerate}
        \item $F^*\colon\Ho^n(\Dc; G^*D) \rightarrow \Ho^n(\CC; F^*(G^*D))$.
        \item $G^*\colon\Ho^n(\E; D) \rightarrow \Ho^n(\Dc; G^*D)$.
        \item $(G\circ F)^*\colon\Ho^n(\E;  D) \rightarrow \Ho^n(\CC; (G \circ F)^*D)$.
    \end{enumerate}
    Observe that the natural systems $F^*(G^*D)$, $(F^*\circ G^*)D$ and $(G \circ F)^*D$ are the same as a consequence of functoriality. Then $(G\circ F)^*=F^* \circ G^*$.
\end{prop}
\begin{proof}
First, we have 
$$(G \circ F)^*D= D \circ \widehat{(G \circ F)}= D \circ \widehat{G} \circ \widehat{F}=G^*D \circ (\widehat{F})=F^*( G^*(D)).$$ Take any chain $(\lambda_1,\dots,\lambda_n)$ in $N_n(\CC)$ and any cochain $f \in F^n(\E;D)$. Then 
\begin{align*}
    F^*(G^*(f))(\lambda_1,\dots,\lambda_n))&=G^*(f)(F(\lambda_1)\circ \cdots \circ F(\lambda_n))\\&=f((G \circ F)(\lambda_1),\dots,(G \circ F)(\lambda_n)).
\end{align*} So the morphisms in cohomology are also the same.
\end{proof}

\subsection{Relative Cohomology}
Let $\CC$ be a small category, $D$ a natural system on $\CC$ and let $\iota\colon\Uc \hookrightarrow \CC $ be a subcategory of $\CC$. Recall that we have a natural system $\iota^*D$ in $\Uc$ by the composition $D \circ \hat{\iota}$. 
For each integer $n\geq 0$  the inclusion induces a morphism  
$\iota^{n*}\colon F^n(\CC;D) \rightarrow F^n(\Uc,\iota^*D)$.

\begin{defi}
    For each non-negative integer $n\geq 0$ we define the group of cochains of $(\CC;D)$ relative to $\Uc$, denoted by $F^n(\CC, \Uc; D)$, as the kernel of the morphism $\iota^{n*}$. 
\end{defi}

\begin{prop}The groups $F^n(\CC, \Uc; D)$ form a cochain subcomplex of $F^{n}(\CC;D)$. 

\end{prop}
\begin{proof}

We must only prove  that if $f$ is in $F^{n-1}(\CC, \Uc{;} D)$ then $\delta(f)$ is in $F^{n}(\CC, \Uc{;} D)$. Take any $n$-chain $(\lambda_1,\ldots,\lambda_{n})$ in $\Uc$, then the $n-1$ chains $(\lambda_2,\ldots,\lambda_{n})$, $(\lambda_1,\ldots,\lambda_i \circ \lambda_{i+1},\ldots,\lambda_{n})$ and $(\lambda_1,\ldots,\lambda_{n-1})$ are in $\Uc$ because $\Uc$ is a subcategory, hence, $f$ vanishes in them.
 So 
\begin{align*}
\delta(f)(\lambda_1,\ldots,\lambda_{n})=&
\lambda_{1*}f(\lambda_2,\ldots,\lambda_n)\\&+ \sum_{i=1}^{n-1} (-1)^i f(\lambda_1, \dots, \lambda_i \circ \lambda_{i+1}, \ldots, \lambda_{n})\\&+  (-1)^{n}\lambda_n^* f(\lambda_1,\ldots,\lambda_{n-1})\\=&\lambda_{1*}(0)+(-1)^{n}\lambda_{n}^*(0)=0.\qedhere
\end{align*}
\end{proof}

\begin{defi}\label{Relative_Cohomology}
 We define the {\em cohomology of $\CC$ relative to $\Uc$} with coefficients in $D$, denoted by $\Ho^\bullet(\CC, \Uc; D)$, as the cohomology of the cochain complex $(F^\bullet(\CC, \Uc; D),\delta)$.
\end{defi}

By definition of relative cohomology, there is a long exact sequence
\begin{equation}\label{SHORT}
\cdots \to \Ho^n(\CC,\Uc;D) \xrightarrow{\gamma} \Ho^n(\CC;D) \xrightarrow{i^*} \Ho^n(\Uc;D)\to \cdots 
\end{equation}
where $i^*$ is the morphism in cohomology induced by the inclusion $i\colon \Uc \rightarrow \CC$. 

\subsection{Cup product}
   Let $D, D', D''$ be three natural systems in the category $\CC$. A {\em pairing} $\mu \colon(D;D') \rightarrow D''$ is a natural way to associate  each $2$-chain $(\lambda_1,\lambda_2)$ with a homomorphism $$\mu_{(\lambda_1,\lambda_2)}\colon D'(\lambda_1) \otimes D''(\lambda_2) \rightarrow D(\lambda_1 \circ \lambda_2). $$ We will denote by $x \cdot y$ the element $\mu(x \otimes y)$. Moreover, the association is natural in the sense that it satisfies the following three identities (\cite[Def. 5.19]{BAUES0}):
    \begin{enumerate}
        \item $(g^*x) \cdot y=x \cdot (g_*y)$,
        \item $f_*(x\cdot y)=(f_*x)\cdot y$,
        \item $h^* (x \cdot y)= x \cdot (h^*y)$.
    \end{enumerate}

    \begin{rem}
    If the three natural systems are the same system $D$ the pairing $\mu\colon (D;D) \rightarrow D$ is called an {\em endopairing}.
\end{rem}
     
Consider a pairing $\mu\co (D;D') \rightarrow D''$. The \emph{absolute cup product}
    $$\smile \colon \Ho^n(\CC; D) \otimes \Ho^m(\CC;D') \rightarrow \Ho^{n+m}(\CC;D'')$$
  is defined by first defining the homomorphism on cochains that takes $f \otimes g$ into $f \smile g$  given by
    $$(f \smile g) (\lambda_1,\dots,\lambda_n, \lambda_{n+1},\dots,\lambda_{n+m}) = f(\lambda_1,\dots,\lambda_n) \cdot g (\lambda_{n+1},\dots,\lambda_{n+m})$$ and then defining a morphism in cohomology:
    $\{f \} \smile \{g\}= \{f \smile g\}$ where $\{f\}$ denotes the cohomology class of the cocycle $f$. This is well defined since we have the following lemma.
  \begin{prop}[{\cite[Lemma 5.20]{BAUES0}}]\label{Cup-product}
  The cup product satisfies:
    \begin{equation}\label{eq:cup_product}
        \delta(f \smile g)= (\delta f) \smile g + (-1)^n f \smile (\delta g)
    \end{equation}
    where $f$ is an $n$-cochain.
\end{prop}

By an inductive argument, we can extend Identity (\ref{eq:cup_product}).  
\begin{prop}Let $f_1, \dots , f_n$ be cochains of degrees $m_1, \dots, m_n$ respectively, then:
\begin{equation}\label{Product_Induction}
    \delta(f_1 \smile \dots \smile f_n)= \sum_{i=1}^{n} (-1)^{N_i} f_1 \smile \dots \smile (\delta f_i) \smile \dots \smile f_n
\end{equation}
where $N_1=0$ and $N_{i+1}=\sum_{j=1}^{i}m_j, 1\leq i \leq n$.
\end{prop}
. 
Now we will prove that the cup product behaves well when we have a functor $F \colon \CC \rightarrow \Dc$ between two categories. 

\begin{prop}
    Let $F \colon \CC \rightarrow \Dc$ be a functor between two small categories and let $D_1,D_2,D_3$ be three natural systems in $\Dc$ with a pairing $\mu: (D_1,D_2) \rightarrow D_3$. We have a pairing $F^*\mu: (F^*D_1,F^*D_2) \rightarrow F^*D_3$. Let  
    $$F_i^*\colon  \Ho(\Dc;D_i) \rightarrow \Ho(\CC,F^*D_i), \quad i=1,2,3,$$
    be the morphisms induced by $F$ in cohomology.
    We claim that 
    $$F_3^*(\{f\} \smile \{g\})=F_1^* (\{f\}) \smile F_2^* (\{g\})$$
    for the absolute cup product determined by $F^*\mu$.
\end{prop}

\begin{proof}
    The existence of the pairing $F^*\mu$ follows from the fact that if   $(\lambda_1,\lambda_2) \in N_2(\CC)$ is a $2$-chain then we have a morphism $$F^*\mu_{(\lambda_1,\lambda_2)}: F^*D_1(\lambda_1) \otimes F^*D_2(\lambda_2) \rightarrow F^*D_3(\lambda_1\circ \lambda_2) $$
    given by $$F^*\mu_{(\lambda_1,\lambda_2)}=\mu_{(F(\lambda_1),F(\lambda_2))}.$$
    
    Take any elements $f \in F^n(\Dc,D_1)$ and $g \in F^m(\Dc,D_2)$, we have:
    \begin{align*}
    &F_3^*(f \smile g)(\lambda_1,\ldots,\lambda_{n+m})\\&=(f \smile g)(F(\lambda_1),\ldots,F(\lambda_{n+m})) \\& = f(F(\lambda_1),\ldots,F(\lambda_{n}))\cdot g(F(\lambda_{n+1}),\ldots,F(\lambda_{n+m}))\\&= F_1^*f(\lambda_1,\ldots,\lambda_{n}) \cdot F^*_2g(\lambda_{n+1},\ldots,\lambda_{n+m}).\qedhere
    \end{align*}
\end{proof}

\subsection{Relative cup product and cup-length}

A family of subcategories $\{\Uc_0, \ldots, \Uc_n \}$ of the category $\CC$ is a {\em geometric cover}  if for every chain in $\CC$  there is some $i\in \{0,\ldots,n\}$ such that the chain belongs to $\Uc_i$.

Recall the definition of the union of two subcategories. For $\Uc$ and $\Vc$ two subcategories of a category $\CC$ we define the subcategory $\Uc \cup \Vc$ as the smallest subcategory of $\CC$ that contains all the objects and morphisms of both $\Uc$ and $\Vc$, as well as all their compositions.

\begin{defi}\label{defi:cup_product_relative}
    Let $\CC$ be a category, $\Uc$ be a subcategory and $\{\Uc_0, \dots, \Uc_n\}$ a geometric cover of $\Uc=\Uc_0 \cup \ldots \cup \Uc_n$. Let $D$ be a natural system in $\CC$ with an endopairing $\mu \colon (D,D) \rightarrow D$. We define \begin{equation}\label{eq:relative_cup_product_chain_complexes}
    \smile\co F^{p_0}(\CC,\Uc_0;D) \otimes \cdots \otimes F^{p_n}(\CC,\Uc_n;D) \rightarrow F^{p}(\CC,\Uc;D)
    \end{equation}
    (where $p=p_0+\dots+p_n$) by $$(f_0 \smile \dots \smile f_n)(\lambda_1,\dots,\lambda_{p})=f_0(\lambda_1,\dots,\lambda_{p_{0}}) \cdots f_n(\lambda_{P_{n}},\dots,\lambda_{p})$$
    where $P_{i}$ is the sum $p_0+\cdots + p_{i-1}$.
    This induces a \emph{relative cup product}
\begin{equation}\label{eq:relative_cup_product_cohomology}
        \smile\co \Ho^{p_0}(\CC,\Uc_0;D) \otimes \cdots \otimes \Ho^{p_n}(\CC,\Uc_n;D) \rightarrow \Ho^{p}(\CC,\Uc;D).
    \end{equation}  
\end{defi}

\begin{prop}\label{rel-cup}\label{VARIOS}
    The products defined in Equations (\ref{eq:relative_cup_product_chain_complexes}) and (\ref{eq:relative_cup_product_cohomology}) in Definition \ref{defi:cup_product_relative} are well-defined. 
\end{prop}

\begin{proof}
    The morphism (\ref{eq:relative_cup_product_chain_complexes})
    is well defined since we have a geometric cover. Indeed, the cochain $f_0 \smile \dots \smile f_n$ vanishes in every chain in $\Uc$. By definition, for any $p$-chain $(\lambda_1,\ldots,\lambda_{p})$ we have that
    $$(f_0 \smile \dots \smile f_n)(\lambda_1,\dots,\lambda_{p})=f_0(\lambda_1,\dots,\lambda_{p_0})  \cdots f_n(\lambda_{P_n+1},\dots,\lambda_{p}).$$ But since the chain $(\lambda_1,\ldots,\lambda_{p})$ is in $\Uc$ and we have a geometric cover, the chain belongs to some $\Uc_i$. Hence the subchain $(\lambda_{P_i+1},\dots,\lambda_{P_{i+1}})$ is in $\Uc_i$, so $f_i(\lambda_{P_i+1},\dots,\lambda_{P_{i+1}})=0$ and the cup-product also vanishes.
    
    Now we will see that morphism \eqref{eq:relative_cup_product_cohomology} is  well-defined.  Formula \eqref{Product_Induction} ensures after a tedious computation that if we have two sequences $f_0, \dots, f_n$ and $f'_0, \dots, f'_n$ such that $\{f_i\}=\{f'_i\}$ then $$\{f_0 \smile \dots \smile f_n\}=\{f'_0 \smile \dots \smile f'_n\}.$$ 
    In the case of two elements the proof goes as follows. 
    
    If we have four cochains $f, f' \in F^n(\CC,\Uc_0;D)$ and $g, g' \in F^n(\CC,\Uc_1;D)$ such that $\{f\}=\{f'\}$ and $\{g\}=\{g'\}$, we will prove that $\{f \smile g\}=\{f'\smile g'\}$. 
    
    As we know $f=f'+\delta h_1$ and $g=g'+\delta h_2$ with $h_1 \in F^{n-1}(\CC,\Uc_0;D)$ and $h_2 \in F^{m-1}(\CC,\Uc_1;D)$, by using the fact that $\delta$ is a derivation (Formula \eqref{eq:cup_product}) we have 
    \begin{align*}
         &(f \smile g)(\lambda_1,\ldots,\lambda_n, \lambda_{n+1}, \ldots \lambda_{n+m})\\
         &=f(\lambda_1,\ldots,\lambda_n) \cdot g(\lambda_{n+1},\ldots,\lambda_{n+m}) \\
         &=(f'+\delta h_1)(\lambda_1,\ldots,\lambda_n) \cdot (g'+\delta h_2)(\lambda_{n+1},\ldots,\lambda_{n+m}) \\
         &= f'(\lambda_1,\ldots,\lambda_n)\cdot g' (\lambda_{n+1},\ldots,\lambda_{n+m})\\ &\quad + f'(\lambda_1,\ldots,\lambda_n) \cdot \delta h_2 (\lambda_{n+1},\ldots,\lambda_{n+m}) \\ &\quad + \delta h_1(\lambda_1,\ldots,\lambda_n) \cdot g' (\lambda_{n+1},\ldots,\lambda_{n+m})\\
        & \quad+ \delta h_1 (\lambda_1,\ldots,\lambda_n) \cdot \delta h_2(\lambda_{n+1},\ldots,\lambda_{n+m}) \\
          &= (f' \smile g' + f' \smile \delta h_2 + \delta h_1 \smile g' + \delta h_1 \smile h_2) (\lambda_1,\ldots\lambda_{n+m}).
    \end{align*}
    But we also have that 
    \begin{enumerate}
        \item $\delta(f \smile h_2)=\delta f \smile h_2+(-1)^{n}f \smile \delta h_2=(-1)^{n}f \smile \delta h_2$, since $\delta f=0$;
        \item  $\delta(h_1 \smile g)=\delta h_1 \smile g+(-1)^{n-1} h_1 \smile \delta g= \delta h_1 \smile g$, since $\delta g=0$;
        \item $\delta(h_1 \smile \delta h_2)= \delta h_1 \smile \delta h_2 +(-1)^{n-1} h_1 \smile \delta  \delta h_2 =\delta h_1 \smile \delta h_2$, since $\delta \delta =0$.
    \end{enumerate}
    Therefore $$f \smile g= f'\smile g' +\delta((-1)^{n}f \smile h_2 + h_1 \smile g + h_1 \smile \delta h_2).\qedhere$$
\end{proof}

Proposition \ref{VARIOS} will be needed in the main result of this paper (Theorem \ref{thm:main_thm}) where geometric coverings with $n+1$ subcategories are considered. A two-factor cup product would need  $\{U_0,\dots, U_i\}$ to be a geometric covering of $U_0\cup \cdots\cup U_i$ for each $i\leq n$.

\begin{lemma}\label{Natural}
    Suppose that we have a geometric cover $\{\Uc_0, \ldots, \Uc_n\}$ of $\Uc=\Uc_0 \cup \ldots \cup \Uc_n$. 
    Let us consider the morphisms 
    $$\gamma_i \colon\Ho^n(\CC,\Uc_i;D) \rightarrow \Ho^n(\CC;D), \quad i=0,\dots,n,$$
   and
    $$\gamma \colon \Ho^n(\CC,\Uc;D) \rightarrow \Ho^n(\CC;D),$$
    as in the exact sequence \eqref{SHORT}. We claim that
    $$\gamma(\{f_0 \smile \dots \smile f_n\})=\gamma_0(\{f_0\}) \smile \dots  \smile \gamma_n (\{f_n\}).$$
\end{lemma}
\begin{proof}
 
We know by the definition of $\gamma$ that $\gamma(\{f_0 \smile \dots \smile f_n\})$ is the cohomology class of the cochain  $f_0 \smile \dots \smile f_n$,  taken in $\Ho^{\bullet}(\CC;D)$. From the definition of the relative cup product and the fact that it is well defined in cohomology, we have that $$\{f_0 \smile \dots \smile f_n\}=\{f_0\}\smile \dots \smile \{f_n\}.$$ Finally, for every $i\in\{0,\ldots, n\}$ we have that $\gamma_i\{f_i\}=\{f_i\}$ where the right part of the equality is the class defined in the total cohomology. Then the proof is complete because
\begin{align*}
        \gamma(\{f_0 \smile \dots \smile f_n\})&=\{f_0 \smile \dots \smile f_n\}=\{f_0\} \smile \dots \smile \{f_n\}\\&=\gamma_0(\{f_0\})\smile \dots \smile \gamma_n(\{f_n\}). \qedhere
    \end{align*}
\end{proof}

\begin{defi}
    Let $\CC$ be a category with a natural system $D$ and an endopairing $\mu \colon (D;D) \rightarrow D$. We define the {\em cup-length} of the cohomology $\Ho(\CC;D)$ as the largest integer $n\geq 1$ such that there are $n$ elements $\xi_1,\ldots,\xi_n$ in $\Ho(\CC;D)$ of degree greater than or equal to $1$ such that $\xi_1 \smile \cdots \smile \xi_n \neq 0$. 
    
    In a similar way, we can define the cup-length of any subset of $\Ho(\CC;D)$.
\end{defi}

\subsection{A computational improvement}

We develop a computationally cheaper way to compute the Baues-Wirsching cohomology of a small category (Definition \ref{def:cohomo_Baues}). More precisely, we  show how to obtain the cohomology from two different projective resolutions of the constant functor $\mathbb{Z}\colon \F \CC \rightarrow \Ab$ (see Definition \ref{Projective_Resolution1} and Corollary \ref{Projective_Resolution2}). On the one hand, by applying a derived functor to the first resolution, we recover the cochain complex $(F^*(\CC;D),\delta)$ defined by Baues and Wirsching. On the other hand, by applying the derived functor to the second resolution, we obtain a \textit{reduced} cochain complex $\tilde{F}^*(\CC;D),\tilde{\delta})$ satisfying $\Ho^*(\tilde{F}^*(\CC;D),\tilde{\delta})\cong \Ho^*(F^*(\CC;D),\delta)$.  The reduced cochain complex is the subcomplex of the cochain complex $F^*(\CC;D)$ given by the cochains $f$ such that $f(\lambda_1,\dots,\lambda_n)=0$ if $\lambda_i=\id$ ($\lambda_i$ is an identity) for any $i\in \{1,\ldots,n\}$. An element $f\in F^*(\CC;D)$ will be referred to as a \textit{non-degenerate} cochain. 

Recall that for  a small category $\CC$, the category of functors $[\F \CC,\Ab]$ is an abelian category with enough projectives and injectives \cite[Section 1.1]{Cegarra}. It is known that for every natural system $D\colon \F \CC \rightarrow \Ab$ we have that $\varprojlim D \simeq \mathrm{Nat}(\mathbb{Z},D)\simeq \Ho^0(\CC,D)$ \cite[Section 1.2]{Cegarra}. Hence, we can define the right derived functor with an injective resolution: $$0 \rightarrow D \rightarrow I_0 \rightarrow \dots \rightarrow I_n \rightarrow \dots $$
and compute the cohomology of the chain complex obtained by applying $\mathrm{Nat(\mathbb{Z},-)}.$ 

Dually, we can begin with a projective resolution of $\mathbb{Z}$: 
$$0 \leftarrow \mathbb{Z} \leftarrow P_0 \leftarrow \dots \leftarrow P_n \leftarrow \dots,$$
then we can apply $\mathrm{Nat}(-,D)$ to obtain a cochain complex, and finally, we can compute its cohomology.

\begin{defi}\label{Projective_Resolution1}
    For each non-negative integer $n$, we define $P_n\colon \F\CC \to \Ab$ as:
    $$P_n=\bigoplus_{\lambda=\lambda_1\circ \dots \circ \lambda_n,(\lambda_1, \dots, \lambda_n) \in\N_n(\CC) } \mathbb{Z}  \mathrm{Hom}(\lambda,-),$$
    where we denote by $ \mathbb{Z} S$ the free abelian group generated by a set $S$. 
\end{defi}

\begin{rem}
    For each $\mu\in \F \CC$, $\mathbb{Z} \mathrm{Hom}(\lambda,\mu)$ is the free abelian group generated by all factorizations $(\alpha,\beta)$ of $\mu$ by $\lambda$, i.e., all $\alpha$ and $\beta$ such that $\mu=\alpha \circ \lambda \circ \beta$.
\end{rem}
\begin{rem}
    For each morphism $\mu$ there is an isomorphism between $P_n(\mu)$ and $\mathbb{Z} N_{n+2}(\CC)(\mu)$ where $$N_{n}(\CC)(\mu)=\{(\lambda_1,\dots,\lambda_n) \in \N_n(\CC)| \mu=\lambda_1 \circ \dots \circ \lambda_n)\}.$$ 
    Indeed for each element $\mu$ we have that 
    
    \begin{align*}
        P_n(\mu)&=\bigoplus_{\lambda=\lambda_1\circ \dots \circ \lambda_n,\lambda_1, \dots, \lambda_n \in\N_n(\CC) }\mathbb{Z} \mathrm{Hom}(\lambda,\mu)\\&=\bigoplus_{\mu=\alpha \circ \lambda_1\circ \dots \circ \lambda_n \circ \beta,(\alpha,\lambda_1, \dots, \lambda_n,\beta) \in\N_{n+2}(\CC) }\mathbb{Z}\\&\simeq \mathbb{Z} \N_{n+2}(\CC)(\mu)
    \end{align*}
\end{rem}

\begin{defi}\label{Natural_Transformation}
    We define a natural transformation $\delta_n \colon P_n\rightarrow P_{n-1}$ term by term. For each morphism $\mu$ we define $\delta_n^\mu \colon P_n(\mu)\rightarrow P_{n-1}(\mu)$ given by:
    $$\delta^\mu_n(\lambda_0,\dots,\lambda_{n+1})=\sum_{i=0}^{n} (-1)^i (\lambda_0,\dots,\lambda_i \circ \lambda_{i+1}, \dots, \lambda_{n+1}).
    $$
\end{defi}

\begin{prop} It holds that
$$0 \leftarrow \mathbb{Z} \leftarrow P_0 \leftarrow \dots \leftarrow P_n \leftarrow \dots$$ as previously defined (Definitions \ref{Projective_Resolution1} and \ref{Natural_Transformation}) is a projective resolution of the constant functor $\mathbb{Z}$.
\end{prop}

 \begin{proof}
     There is a contracting chain homotopy $s^\mu_n\colon P_n(\mu) \rightarrow P_{n+1}(\mu)$ given in each $\mu$ as follows:
     $s^\mu_n(\lambda_0,\dots,\lambda_{n+1})=(\id,\lambda_0,\dots,\lambda_{n+1})$
     \cite{BAUES1}.
 \end{proof}
 
Now we will see how to recover the original definition of Baues-Wirsching cohomology from this construction.
\begin{prop}
Let $D\colon \F \CC \rightarrow \Ab$ be a natural system. For each $\lambda\colon c \rightarrow d$ there is natural isomorphism $\mathrm{Nat}( \mathbb{Z} \mathrm{Hom}(\lambda,-),D)\simeq D(\lambda).$
\end{prop}

\begin{proof}
        Let $\alpha \colon \mathbb{Z} \mathrm{Hom}(\lambda,-) \implies D$ be a natural transformation. We can define an element in $D(\lambda)$ by the image of the neutral element of $(\id_c,\id_d) \in \mathbb{Z} \mathrm{Hom}(\lambda,\lambda)$, by $\alpha$, i.e., $\alpha_{\lambda}(\id_d,\id_c) $. 
        
        Reciprocally, suppose that we take an element $v \in D(\lambda)$. We can define a natural transformation $\alpha$ given by $\alpha_\lambda(\id_d,\id_c)=v$. Indeed, we can naturally extend this to a natural transformation. For each $\mu$ there is a morphism $\alpha_\mu \colon  \mathbb{Z} \mathrm{Hom}(\lambda,\mu) \rightarrow D(\mu)$ given by taking any $(\gamma,\beta) \colon \lambda \rightarrow \mu$ to the element $\alpha_\mu(\gamma,\beta)=D(\gamma,\beta)(v).$ This is obviously a natural transfomation by the functoriality of $D$.
\end{proof}    

\begin{prop}\label{Natural_Transformations_Cochains}
    Let $D\colon \F \CC \rightarrow \Ab$ be a natural system. We have that: $\mathrm{Nat}(P_n,D)\simeq F^n(\CC;D).$
\end{prop}
\begin{proof}This follows from a chain of equivalences:
\begin{align*}
    \mathrm{Nat}(P_n,D)&\simeq \mathrm{Nat}(\bigoplus \mathbb{Z} \mathrm{Hom}(\lambda,-),D)\\& \simeq \prod \mathrm{Nat}( \mathbb{Z} \mathrm{Hom}(\lambda,-),D) \simeq \prod D(\lambda) \simeq F^n(\CC;D)
\end{align*}
     where we always assume that the coproduct and product is indexed by all $n$ chains $(\lambda_1, \dots , \lambda_n)$ and $\lambda=\lambda_1\circ \dots \circ \lambda_n$.
\end{proof}

Moreover we can obtain a reduced cochain complex. 
\begin{defi}
    Let $n$ be a strictly positive integer. We define the natural system $P_n'$ as: $$P_n'=\bigoplus_{\lambda=\lambda_1\circ \dots \circ \lambda_n,\lambda_1, \dots, \lambda_n \in\N'_n(\CC) } \mathbb{Z} \mathrm{Hom}(\lambda,-),$$ where $\N'_n(\CC)$ is the \textit{degenerate nerve}, i.e., $\N'_n(\CC)$ is the set of chains $(\lambda_1,\dots,\lambda_n)$ such that there is some $i\in\{1,\dots,n\}$ with $\lambda_i=\id$, i.e., being an identity. We also define $P'_0$ as the constant functor $0$.
\end{defi}

\begin{rem}
    There is also a nice expression of the previous chain complex if we evaluate in a morphism $\mu$. \begin{align*}
    P'_n(\mu)=\bigoplus_{\lambda=\lambda_1\circ \dots \circ \lambda_n,(\lambda_1, \dots, \lambda_n) \in\N'_n(\CC) } \mathbb{Z} \mathrm{Hom}(\lambda,\mu)=\\ \bigoplus_{\mu=\alpha \circ \lambda_1\circ \dots \circ \lambda_n\circ\beta,(\lambda_1 \dots, \lambda_n) \in\N'_n(\CC) } \mathbb{Z}.\end{align*} Hence, $P'_n(\mu)$ is the free abelian group of all $n+2$-chain $(\lambda_0,\dots,\lambda_{n+1})$ such that there is at least one interior arrow in the chain such that is an identity, i.e., $\lambda_i=\id$ for $0 < i <n+1$.
\end{rem}
    
\begin{prop}
    Restricting the natural transformations $\delta_n$ from Definition \ref{Natural_Transformation} we have a projective resolution of the constant functor $0.$
\end{prop}
\begin{proof}
    First we prove that we can restrict the natural transformations. Indeed, take a morphism $\mu$ and chain $(\lambda_0,\dots,\lambda_{n+1})$ such that $\mu=\lambda_0\circ \dots \lambda_{n+1}$ and there is $0 < i <n+1$ such that $\lambda_i=\id$. We have that $\delta^\mu_n(\lambda_0,\dots,\lambda_{n+1})$ is the alternating sum of compositions. If we have that $k\neq i-1,i$ then the chain $(\lambda_0,\dots,\lambda_k \circ \lambda_{k+1},\dots,\lambda_n)$ is also degenerate since it contains the morphism $\lambda_i$. If $k=i-1$ or $i$ we have that $(\lambda_0,\dots, \lambda_{i-1} \circ \id, \lambda_i,\dots, \lambda_{n+1})=(\lambda_0,\dots, \lambda_{i-1},\lambda_i,\dots,\lambda_n)=(\lambda_0,\dots,\lambda_{i-1},\id \circ \lambda_{i},\dots,\lambda_{n+1})$ and the alternate sum of them is $0$.

    Now we can see that is also a projective resolution since the homotopy defined by $s^\mu_n(\lambda_0,\dots,\lambda_{n+1})=(\id,\lambda_0,\dots,\lambda_{n+1})$ takes a degenerate chain to a degenerate chain and we also have the desired result. Furthermore, since $P_0'=0$ we have a resolution of $0$. 
\end{proof}

\begin{defi}
    We define the functor $\tilde{P}_n$ as the quotient $\frac{P_n}{P'_n}$. More precisely, for each $\lambda$ we have $\tilde{P}_n(\lambda)=\frac{P_n(\lambda)}{P'_n(\lambda)}$.
\end{defi}

\begin{rem}
     $\tilde{P}_n(\lambda)$ is the abelian group generated by all chains  $(\lambda_0,\dots,\lambda_{n+1})$  of length $n+2$ such that $\lambda=\lambda_0 \circ \dots \circ \lambda_{n+1}$ and $\lambda_i\neq\id$ for any $0 < i <n+1$
\end{rem}
Using the previous Proposition we obtain that this defines a new projective resolution of $\mathbb{Z}$.
\begin{corollary}
    There is a projective resolution of $\mathbb{Z}$ given by $\tilde{P}_n$.
\end{corollary}

With this we can define the wanted reduced cochain complex.
    
\begin{corollary}\label{Projective_Resolution2}
    We can define the cohomology of a small category with a natural system $D \colon \F \CC \rightarrow \Ab $ with the previous projective resolution. This means that $\Ho^*(\CC,D)$ is the cohomology of the cochain complex given by $\mathrm{Nat}(\tilde{P}_n,D)=\tilde{F}^*(\CC;D)$ where $\tilde{F}^*(\CC;D)$ is the abelian group of non-degenerate cochains, i.e., it only takes nonzero values in the non-degenerate cochains.
\end{corollary}

\begin{proof}
    It follows from the fact that $\tilde{P}_n$ can be expressed as $$\bigoplus_{\lambda=\lambda_1\circ \dots \circ \lambda_n,(\lambda_1, \dots, \lambda_n) \in\tilde{\N}_n(\CC) } \mathbb{Z} \mathrm{Hom}(\lambda,\mu)$$ where $\tilde{\N}(\CC)$ denotes the nontrivial chains and the group of natural transformations between $\tilde{P}_n$ and $D$ is precisely $\tilde{F}^*(\CC;D)$ using a similar argument as the one given in the proof of Proposition \ref{Natural_Transformations_Cochains}.
\end{proof}

Furthermore, it can be shown how the relative cohomology constructions and the cup product also work for this definition of cohomology.

The abelian group morphism $\iota^{n*} \colon F^n(\CC;D) \rightarrow F^n(\Uc;D)$ can be restrited to a morphism $\tilde\iota^{n*} \colon \tilde F^n(\CC;D) \rightarrow \tilde F^n(\Uc;D)$. We will also have a cochain subcomplex $\tilde F^*(\CC,\Uc;D)$ given by the kernel of the previous maps. Moreover the cohomology groups of this cochain complex will be the same as the one defined in Definition \ref{Relative_Cohomology}.

\begin{prop}
    The cohomology groups of cochain complexes $F^*(\CC,\Uc;D)$ and $\tilde F^*(\CC,\Uc;D)$ are the same.
\end{prop}
\begin{proof}
    We have a commutative diagram of cochain complex
\[\begin{tikzcd}
	{F^*(\CC,\Uc;D)} & {} & {F^*(\CC;D)} && {F^*(\Uc;D)} \\
	{\tilde F^*(\CC,\Uc;D)} && {\tilde F^*(\CC;D)} && {\tilde F^*(\Uc;D)}
	\arrow["\gamma", from=1-1, to=1-3]
	\arrow["{\iota^*}", from=1-3, to=1-5]
	\arrow[from=2-1, to=1-1]
	\arrow["{\tilde \gamma}"', from=2-1, to=2-3]
	\arrow[from=2-3, to=1-3]
	\arrow["{\tilde \iota}"', from=2-3, to=2-5]
	\arrow[from=2-5, to=1-5]
\end{tikzcd}\]

where the horizontal sequences are short exact sequences, the at rightmost
vertical arrows are the inclusions that take a non-degenerate cochain to a cochain that vanishes in degenerate chains and the left vertical arrows is induce by the universal property of the kernel. Since the two leftmost vertical inclusions induce cohomology isomorphisms, the last one also induces a cohomology isomorphism by Five lemma in the long exact sequence in cohomology given by the two short exact sequences.
\end{proof}

In turn, the cup product is also well-defined since if we take two non-degenerate cochains $f\in \tilde F^n(\CC;D)$ and $g \in \tilde F^m(\CC;D)$ and degenerate chain $(\lambda_1,\dots,\lambda_{n+m})$ where $\lambda_i=\id$ for some $i$ then
$$f\smile g (\lambda_1,\dots,\lambda_{n+m})=f(\lambda_1,\dots,\lambda_{n}) \cdot g(\lambda_{n+1},\dots,\lambda_{n+m})=0$$
since $f(\lambda_1,\dots,\lambda_{n})=0$ if $i\leq n$ and $g(\lambda_{n+1},\dots,\lambda_{n+m})=0$ if $i\geq n+1$.

\subsection{An example}\label{subsec:example}
We illustrate how to compute Baues-Wirsching cohomology and the cup-length for a specific category and natural system. 
    Let $\Si$ be the small category depicted in  the following diagram (identities are omitted):
    $$\begin{tikzcd}[column sep=12pt, row sep=12pt]    
C \arrow[r, "\alpha", bend left] \arrow[r, "\beta"', bend right] & D
\end{tikzcd}
$$
    Its factorization category $\mathcal{FS}$ is generated by the following diagram:
    $$
\begin{tikzcd}[column sep=12pt, row sep=12pt] 
 & \alpha &\\
\id_C \arrow[rd] \arrow[ru] &        & \id_D \arrow[lu] \arrow[ld] \\
& \beta  &                            
\end{tikzcd}$$

An example of a natural system $\mathcal{D}$ in the category $\Si$ is:
    $$
\begin{tikzcd}
                                              & \mathbb{Z}   &                                               \\
\mathbb{Z} \arrow[rd, "p"'] \arrow[ru, "\id"] &              & \mathbb{Z} \arrow[lu, "\id"'] \arrow[ld, "p"] \\
                                              & \mathbb{Z}_2 &                                              
\end{tikzcd}
$$
where $p\colon \mathbb{Z} \rightarrow \mathbb{Z}_2$ is the projection.

We will show that the cohomology groups are 
$$\Ho^0(\Si;D)=\mathbb{Z}, \quad \Ho^1(\Si;D)=\mathbb{Z}_2 \text{\ and\ } \Ho^n(\Si;D)=0 \text{\ for\ } n\geq 2.$$ 

We will use the reduced cochain complex. We will denote a $0$-cochain $f \in \tilde{F}^0(\Si,D)$ by $(c,d)$ where $f(\id_C)=c$ and $f(\id_D)=d$. In a similar way we will denote a $1$-cochain $g \in \tilde{F}^1(\Si,D)$ by $(a,b)$ if $g(\alpha)=a$ and $g(\beta)=b$. Since all chains of length $n\geq 2$ are degenerate we have that $\tilde{F}^n(\mathcal{S};D)=0$. Using these notations, the coboundary map is given by
    $$\delta^0(c,d)=(c-d,[c+d])$$
    and
    $$\delta^n=0 \, (n\geq 1).$$
    Hence we have  $$\Ho^0(\Si;D)=\ker \delta^0=\{(c,d) \in \mathbb{Z} \times \mathbb{Z} \colon c=d, [c]=[d]\}=\{(c,c)\in \mathbb{Z}  \times \mathbb{Z}\}\cong \mathbb{Z}.$$
    Also,
    $$\Ima \delta^0=\{(c-d,[c+d]) \colon (c,d) \in \mathbb{Z} \times \mathbb{Z}\} $$ and
    $$\ker\delta^1=\tilde{F}^1(\CC;D),$$
    so, $$\Ho^1(\Si;D)=\frac{\ker \delta^1}{\Ima \delta^0} \cong \mathbb{Z}_2.$$ 
Moreover, we know that $\Ho^n(\Si;D)=0$ for every integer $n\geq 2$.
Therefore, the cup-length of $\Ho(\Si;D)$ is at most $1$ for every endopairing since $\Ho^2(\Si;D)=0$. In particular we know that the cup length of $\Ho(\Si;D)$ is $1$ for the endopairing in which we define the product as the multiplication by an integer with the structure of a bimodule.

Now consider the subcategory generated by the object $C$, i.e., the subcategory with only one object and the identity morphism. We can compute the relative cohomology by using the following facts:
\begin{align*}
\iota^{0*}(c,d)&=(c),\\
 \iota^{1*}(a,b)&=(0),\\
\end{align*}
So we conclude that:
\begin{align*}
\tilde{F}^0(\Si, C ; D)&=\{(0,d)\},\\
\tilde{F}^1(\Si, C ; D)&=\{(a,b)\},\\
\end{align*}
The coboundary maps are given by
\begin{align*}
    \delta^0(0,d)&=(-d,[d]),\\
\end{align*}
Hence, it can be verified that $\Ho^0(\Si,C;D)=0$ and $\Ho^1(\Si,C;D)\cong \mathbb{Z}_2$. 
\section{Bifibrations, \v{S}varc genus and sectional number} \label{sec:bifibrations_Svarc}

In this section, we recall the notion of a bifibration, state and prove some of its properties, and provide some relevant examples. After that, we introduce the notions of \v{S}varc genus and sectional number of a functor, and we prove  some of their properties and the main result of the paper.

\subsection{Homotopy in small categories} We begin by recalling the notion of homotopy in small categories as introduced by Lee (\cite{LEE,LEE2}) and by establishing some notation. 

The {\em interval category $\mathbb{I}_m$} of length $m\geq 0$ consists of $m+1$ objects with zigzag arrows,
	 $$0 \longrightarrow 1 \longleftarrow 2 \longrightarrow \cdots \longrightarrow (\longleftarrow) m.$$

Given two small categories $\CC$ and $\Dc$ we denote their product by $\CC \times \D$. 
	Let $F,G\colon \CC\to \Dc$ be two functors between small categories. We say that $F$  and $G$ are {\em homotopic}, denoted by $F \simeq G$,  if, for some $m\geq 0$, there exists a functor $H\colon \CC \times \mathbb{I}_m \rightarrow \D$, called a homotopy (of length $m$), such that $H_0=F$ and $H_m=G$.

\subsection{The notion of a bifibration}
In this subsection we follow the references  \cite{DAZIN}, \cite[Appendix A]{KOEN}, \cite{STREITCHER}, \cite{VISTOLI} and \cite[Chapter 12]{PENNER}, as well as \cite[Section 4]{TANAKA}. 

Let $P\colon \E \to \Ba$ be a functor.
A morphism $\phi\in \E(e_1,e_2)$ is {\em Cartesian} (with respect to $P$) if, for every arrow $\beta \in\E(e,e_2)$ and  every arrow
 $\bbar\alpha\in\Ba(Pe,Pe_1)$  such that $P(\phi)\circ \bbar\alpha=P\beta$,
there exists a unique arrow $\alpha\in\E(e,e_1)$ such that 
$\phi\circ\alpha=\beta$ and 
$P(\alpha)=\bbar\alpha$ (see Diagram (\ref{diag:def_cartesian})).
\begin{equation}\label{diag:def_cartesian}
\begin{tikzcd}
& e_1 \ar{d}{\phi} \\
 e \arrow[ur,dashrightarrow,"\alpha"]\arrow[r,"\beta"'] & e_2 
\end{tikzcd}
\quad \quad
 \begin{tikzcd}
& P(e_1) \arrow[d,"P\phi"] \\
P(e)\arrow[ur,"{\bbar\alpha}"]\arrow[r,"P\beta"'] & P(e_2)
\end{tikzcd}
\end{equation}

\begin{defi}A functor $P\colon \E \to \Ba$ is a {\em fibration} if for any arrow $\bbar\phi\in \Ba(b_1,b_2)$ and any object $e_2\in\mathrm{Obj}(\E)$ such that $P(e_2)=b_2$, there exists a Cartesian arrow $\phi\in \E(e_1,e_2)$ such that $P(\phi)=\bbar\phi$. The arrow $\phi$ is called a {\em Cartesian lift} of $\bbar\phi$ with codomain $e_2$.
\end{defi}

\begin{defi}
Let $P\colon \E \to \Ba$ be a fibration. We define $\E_b$ as the {\em fiber} of $P$ over $b\in\mathrm{Obj}(\Ba)$, i.e., as the subcategory of $\E$ with objects $e\in \mathrm{Obj}(\E)$  such that $Pe=b$, and with arrows 
$\nu\in \E(e_1,e_2)$ such that $P\nu=\id_b$. These arrows are called {\em vertical} arrows.
\end{defi}

The Cartesian lift of a given $\bbar\phi\colon b_1 \to b_2$ with a given codomain $e_2$ is unique, up to a unique vertical arrow. This follows from the unicity in the definition of Cartesian arrow.

By using the axiom of choice, we can take a particular lift, which will be denoted by
$$\Cart(\bbar\phi, e_2)\colon \bar\phi^*e_2 \to e_2.$$
This particular choice defines a functor ${\bar\phi}^*\colon \E_{b_2} \to \E_{b_1}$, where the image  of a vertical arrow $\nu\in \E_{b_2}$ is given by the unique arrow ${\bbar\phi}^*\nu$ making the following diagram commute:
$$\begin{tikzcd}
\bbar\phi^*e_2\arrow[d,"{\bbar\phi^*\nu}"',dashed]\arrow[r,"{\Cart(\bbar\phi, e_2)}"]&e_2\arrow[d,"\nu"]\\
\bbar\phi^*e_2'\ \ \arrow[r,"{\Cart(\bbar\phi, e_2')}"]&\ \ e_2'.
\end{tikzcd}
$$
It corresponds to the diagram
$$
\begin{tikzcd}[column sep=12pt, row sep=12pt] 
&& {\bbar\phi}^*e_2' \ar{d}{\Cart} \\
{\bbar\phi}^*e_2\arrow[rru,dashrightarrow] \arrow[r] &e_2\arrow[r,"\nu"']& e_2 
\end{tikzcd}
\quad \quad
 \begin{tikzcd}[column sep=12pt, row sep=12pt] 
&& b_1 \arrow[d,"\phi"] \\
b_1\arrow[urr,"{\id}"]\arrow[r,"\phi"'] &b_2\arrow[r,"\id"]& b_2
\end{tikzcd}
$$

The functoriality of ${\bbar\phi}^*$ follows again from unicity. It is called the {\em pullback functor}.

Let $P^\op \colon \E^\op \to \Ba^\op$ be the opposite functor of $P\colon \E \to \Ba$.
\begin{defi}The arrow $\varphi\colon e_1 \to e_2$ in $\E$ is {\em op-Cartesian} for the functor $P$ if the opposite arrow $\varphi^\op$ in $\E^\op$ is Cartesian for $P^\op$. Explicitly, this means that for any given $\beta\in \E(e_1,e)$ and any given $\bbar\alpha\in \Ba(Pe_2,Pe)$ such that $\bbar\alpha\circ P\varphi=P\beta$, there exists a unique $\alpha\in\E(e_2,e)$ such that $\alpha\circ\varphi=\beta$ and $P\alpha=\bbar{\alpha}$, as in the following diagram:
$$\begin{tikzcd}[column sep=12pt, row sep=12pt] 
e_1 \ar[d,"\varphi"']\arrow[r,"\beta"] &e \\
 e_2 \arrow[ur,dashrightarrow,"\alpha"']&  
\end{tikzcd}
\quad \quad
 \begin{tikzcd}[column sep=12pt, row sep=12pt] 
 Pe_1 \arrow[d,"P\varphi"'] \arrow[r,"P\beta"] &Pe \\
Pe_2\arrow[ur,"{\bbar\alpha}"']& 
\end{tikzcd}
$$
\end{defi}

\begin{defi}A functor $P\colon \E \to \Ba$ is an {\em op-fibration} if for any map $\bbar\varphi\colon b_1\to b_2$ in $\Ba$, and for any object $e_1 \in \mathrm{Obj}(\E)$ with $Pe_1=b_1$, there exists an op-Cartesian arrow $\varphi\colon e_1 \to e_2$ such that $P\varphi=\bbar\varphi$.
\end{defi}

Again, this {\em op-Cartesian lifting} is unique up to a unique vertical isomorphism. Then we can choose some particular lifting
$$\opCart(\bbar \varphi, e_1)\colon e_1 \to \bbar\varphi_*e_1$$
so defining a functor
$$\bbar\varphi_*\colon \E_{b_1} \to \E_{b_2}$$
between the fibers.

  \begin{defi}We say that the functor $P\colon \E \to \Ba$ is a {\em bifibration} if it is both a fibration and an op-fibration.
\end{defi}

\begin{rem}
    Observe that the notion of bifibration is more general than the notion of fibration for the model structure on the category of small categories  in \cite{Tanaka_model_structure}. More precisely, our definition allows categories fibered and cofibered over any small category (\cite[Proposition 3.1]{nLab2}), while the one presented in \cite{Tanaka_model_structure} only allows categories fibered and cofibered over groupoids.
\end{rem}

\begin{example}\label{Groupoid}
    Let $\E$ be a groupoid of two elements and $\Ba$ the group $\mathbb{Z}_2$ regarded as categories, i.e, we have the categories:
    $$
\begin{tikzcd}[column sep=12pt, row sep=12pt] 
A \arrow[r, "f", bend left] & B \arrow[l, "g", bend left] &\quad \bullet \arrow["h"', loop, distance=2em, in=125, out=55]
\end{tikzcd}
    $$
    where $h \circ h= \id_{\bullet}$, $f \circ g = \id_A$ and $g 
 \circ f= \id_B$. We define the functor $P \colon \E \rightarrow \Ba$ as the one that takes every object to $\bullet$ and satisfies $P(f)=P(g)=h$. This functor is a bifibration.
\end{example}

\subsection{Homotopy lifting property}

Let $m\in \mathbb{N}$. We define the functor $i_0\colon \CC \to \CC\times \mathbb{I}_m$ which sends each object $c$ to $(c,0)$ and each arrow $f \colon c_1\to c_2$ to $f\times \id_0$.
    The functor $P\colon \E \to \Ba$ satisfies the \emph{ right homotopy lifting property} if for any diagram:
    \begin{equation*}
\begin{tikzcd}
\CC \arrow[d,"i_0"'] \arrow[r,"G"] & \E  \arrow[d,"P"] \\
\CC\times \mathbb{I}_m \arrow[r,"H"']                   & \Ba           
\end{tikzcd}
\end{equation*}
there exists a functor $\wtilde H \colon \CC\times \mathbb{I}_m \to\E$ such that the following diagram is commutative:
\begin{equation*}
\begin{tikzcd}
\CC \arrow[d,"i_m"'] \arrow[r,"G"] & \E  \arrow[d,"P"] \\
\CC\times \mathbb{I}_m \arrow[ur,"\wtilde H",dashed]\arrow[r,"H"']                   & \Ba           
\end{tikzcd}
\end{equation*}

Using Proposition 4.2 and Corollary 4.3 in \cite{TANAKA}, it follows:

\begin{prop}\label{Lifting}
Let $P\colon \E \to \Ba$ be a functor between small categories. It $P$ is a bifibration then it satisfies the right homotopy lifting property. 
\end{prop}

The next result, Proposition \ref{EQUIVALE}, will allow us to speak about the homotopy invariants of ``the fiber'' of a bifibration.

\begin{prop}[{\cite[Proposition 4.4]{TANAKA}}]\label{EQUIVALE}
Let $P\colon\E \rightarrow \Ba$ be a bi-fibration.  If $b_1$ and $b_2$ are two objects in $\Ba$ such that   there is a morphism $u\colon b_1 \to b_2$, then there is a categorical equivalence between the fibers  $\E_{b_1}$ and $\E_{b_2}$.
\end{prop}

\subsection{Pullbacks} We now prove that the pullback of a bifibration is a bifibration.
Let $P\colon \E \to \Ba$ be a bifibration and $F\colon \Ba'\to \Ba$ a functor between small categories. Consider the diagram:
 \begin{equation}\label{eq:diagram_pullback}
     \begin{tikzcd}[column sep=12pt, row sep=12pt] 
& \E \arrow[d, "P"] \\
\Ba' \arrow[r, "F"] & \Ba               
\end{tikzcd}
 \end{equation}
The pullback of $P$ by $F\colon \Ba'\to \Ba$ is the limit in \textbf{Cat} of Diagram (\ref{eq:diagram_pullback}):
\begin{equation*}
    \begin{tikzcd}[column sep=12pt, row sep=12pt] 
\E' \arrow[d, "P'"] \arrow[r, "F'"] & \E \arrow[d, "P"] \\
\Ba' \arrow[r, "F"]                  & \Ba               
\end{tikzcd}
\end{equation*}

  The category $\E'$ is the subcategory of $\Ba' \times \E$ with objects $(C,D)$ such that $F(C)=P(D)$ and morphisms $(f,g)$ such that $F(f)=P(g)$. 

\begin{prop}
    The pullback of a bifibration is a bifibration, i.e., if we have the following diagram:
    $$
\begin{tikzcd}[column sep=12pt, row sep=12pt] 
\E' \arrow[d, "P'"'] \arrow[r, "F'"] & \E \arrow[d, "P"] \\
\Ba' \arrow[r, "F"]                  & \Ba               
\end{tikzcd}
    $$
    with $P$ a bifibration and $\E'$ the pullback of $P$, then $P'$ is also a bifibration.
\end{prop}

\begin{proof}
First, we will prove that $P'$ is an op-fibration.  Take any morphism $f\co C \rightarrow D$ in $\Ba'$ and fix a lift of $C$. This means that we take a pair $(C,X)$ in $\E$. We can lift $f$ to the morphism $(f,\hat{f}\,)$ with $\hat{f}$ a Cartesian lift of $F(f)$ with $X$ as the domain. This can be done since $P(X)=F(C)$. 

We claim that this morphism is a Cartesian lift of $f$. Indeed, take the following situation:
    $$
\begin{tikzcd}
{(C,X)} \arrow[d, "{(f,\hat{f})}"'] \arrow[r, "{(g,l)}"] & {(A,Y)} & C \arrow[d, "f"'] \arrow[r, "g"] & A \\
{(D,E)}                                                  &         & D \arrow[ru, "h"']              &  
\end{tikzcd}
    $$
    We can lift $h$ to $(h,\hat{h}\,)$ where $\hat{h}$ is the lift of $F(h)$ that has the lifting property. That means that we have the following situation: 
    $$
\begin{tikzcd}
X \arrow[d, "\widehat{f}"'] \arrow[r, "l"] & Y & F(C) \arrow[d, "F(f)"'] \arrow[r, "F(g)"] & F(A) \\
E \arrow[ru, "\hat{h}"', dotted]       & \quad  & F(D) \arrow[ru, "F(h)"']                 &     
\end{tikzcd}
$$
Proving that $P$ is a fibration is totally analogous.
\end{proof}

\subsection{Coverings}

In this subsection we prove that coverings of small categories are examples of bifibrations. 

Let $\CC$ be a small category and let $C$ be an object of $\CC$. We define the sets:
    $$\mathcal{U}_C=\{f \in \mathrm{Mor\, } \CC \, \colon \mathrm{Codom}(f)=C\}$$
    and
    $$\mathcal{F}_C=\{f \in \mathrm{Mor\, } \CC \, \colon \mathrm{Dom}(f)=C\}.$$
    In other words, $\mathcal{U}_C$ is the set of morphisms that have $C$ as their target and $\mathcal{F}_C$ is the set of morphisms that have $C$ as its  source .

Any  functor  $P\colon \CC \to \D$ between small categories induces maps $P_{\mathcal{U}_C} \colon \mathcal{U}_C \rightarrow \mathcal{U}_{P(C)}$ and $P_{\mathcal{F}_C} \colon \mathcal{F}_C \rightarrow \mathcal{F}_{P(C)}$. Both maps are defined by the property that if we have a morphism $f$ with domain or codomain $C$ then the morphism $P(f)$ has the object $P(C)$ as its domain or codomain.

\begin{defi}[\cite{BARMIN}]
    A functor $P \colon \E \rightarrow \Ba$ is a {\em covering} if its surjective on objects and for every object $C$ the morphisms $P_{\mathcal{U}_C}$ and $P_{\mathcal{F}_C}$ are bijective.
\end{defi}

\begin{prop}
   Every covering is a bifibration.
\end{prop}

\begin{proof}
    First let us see that $P$ is a fibration. Let $f \colon C \rightarrow D$ be a morphism in $\Ba$.  If we fix an element $\widehat{D}$ such that $P(\widehat{D})=D$, we can lift $f$ using the fact that $P_{\mathcal{U}_{\widehat{D}}}$ is a bijection between $\mathcal{U}_{\widehat{D}}$ and $\mathcal{U}_{P(\widehat{D})}=\mathcal{U}_D$. We will denote that lift of $f$ by $\widehat{f}$ and the source of this map by $\widehat{C}$, so $\widehat{f} \colon \widehat{C} \rightarrow \widehat{D}$.  Now suppose that we have the following situation :
    $$
\begin{tikzcd}[column sep=12pt, row sep=12pt] 
& \widehat{C} \arrow[d, "\widehat{f}"] &                                 & C \arrow[d, "f"] \\
\widehat{E} \arrow[r, "\widehat{g}"'] &\widehat{D}                    & E \arrow[r, "g"'] \arrow[ru, "h"] & D               
\end{tikzcd}
    $$
    We must define $\widehat{h} \colon\widehat{E} \rightarrow\widehat{C}$ such that when we put it in the left diagram we get a commutative diagram and $P(\widehat{h})=h$. The lift $\widehat{h}$ exists because we have a bijection $P_{\mathcal{U}_{\widehat{C}}}$between $\mathcal{U}_{\widehat{C}}$ and $\mathcal{U}_{P(\widehat{C})}=\mathcal{U}_C$. Moreover we have that $\widehat{f} \circ \widehat{h}= \widehat{g}$. On the one hand, we have that $\widehat{f} \circ \widehat{h}= P^{-1}_{\mathcal{U}_{\widehat{D}}}(f) \circ P^{-1}_{\mathcal{U}_{\widehat{C}}}(h)$. On the other hand, if we apply $P$ to the previous term we get:
    $$P(P^{-1}_{\mathcal{U}_{\widehat{D}}}(f) \circ P^{-1}_{\mathcal{U}_{\widehat{C}}}(h))=P(P^{-1}_{\mathcal{U}_{\widehat{D}}}(f)) \circ P(P^{-1}_{\mathcal{U}_{\widehat{C}}}(h))=f \circ h= g.$$ Now we can prove that $\hat{f} \circ \widehat{h}=\widehat{g}$ since the bijection between $\mathcal{U}_{\widehat{D}}$ and $\mathcal{U}_D$ gives us the wanted equality, because $P(\widehat{g})=g=P(\widehat{f} \circ \widehat{h})$.

    Proving that $P$ is also an op-fibration is totally analogous and follows from the fact that we have the bijection $P_{\mathcal{F}_C}$.
\end{proof}

\begin{rem}
Note that the definition of  covering is equivalent  to the one given in \cite{Tanaka_model_structure} for functors surjective on objects since a discrete groupoid is the same as a discrete category.
\end{rem}

\begin{example}\label{Doblecir}
 Let $\Si$ be the category:
    $$
\begin{tikzcd}[column sep=12pt, row sep=12pt] 
C \arrow[r, "\alpha", bend left] \arrow[r, "\beta"', bend right] & D
\end{tikzcd}
    $$
 of Subsection \ref{subsec:example}   and let $\E$ be the category:
    $$
\begin{tikzcd}[column sep=12pt, row sep=12pt] 
& D_1 &     \\
C_1 \arrow[ru, "\alpha_1"] \arrow[rd, "\beta_1"'] &     & C_2 \arrow[lu, "\beta_2"'] \arrow[ld, "\alpha_2"] \\
    & D_2 &                                                  
\end{tikzcd}.
$$

The functor $P \colon \E \rightarrow \Si$ defined by $P(C_i)=C$, $P(D_i)=D$, $P(\alpha_1)=\alpha$ and $P(\beta_i)=\beta$ is a covering  and therefore a bifibration.
\end{example}

\subsection{\v{S}varc genus and sectional number}

The aim of this section is to introduce for the first time the  notions of sectional number and \v{S}varc genus of a functor; to show that they coincide for bifibrations; to prove Theorem \ref{thm:main_thm}  and to show it in an example.

\begin{defi}
    Let $P\co \CC \rightarrow \Dc$ be a functor. We define its \emph{sectional number}, denoted by $\sC (P)$, as  the least integer $n\geq 0$ such that there is a geometric cover $\{\Uc_0, \ldots, \Uc_n \}$ of $\Dc$ by subcategories that admit local sections, that is, such that for every $i \in \{0, \ldots, n\} $ there is a functor $s_i \colon \Uc_i \rightarrow \CC$ which is a local right inverse of $F$,  in the sense that $P \circ s_i = \iota_i$, where $\iota_i \colon \Uc_i \rightarrow \Dc$ is the inclusion.
\end{defi}

\begin{example}\label{Svar_ejemplo}
    In the case of the bifibration defined in Example \ref{Doblecir} there is no global section. Otherwise  we would have $P(C)=C_i$ and $P(D)=D_j$; but if we have a lift of $\alpha$ we can not have a lift for $\beta$ and vice versa. Nonetheless there are local sections for the subcategories:
$$
\begin{tikzcd}
C \arrow[r, "\alpha"] & D 
&\text{and}& C \arrow[r, "\beta"]  & D.
\end{tikzcd}
$$
So we can conclude that $\sC(P)=1$.
\end{example}

\begin{defi}
    Let $F\co \CC \rightarrow \Dc$ be a functor, we define its {\em \v{S}varc genus}, denoted by $\sG (P)$, as the least  integer $n\geq 0$ such that there is a geometric cover $\{\Uc_0, \ldots, \Uc_n \}$ of $\CC$ by subcategories that admit local homotopic sections, i.e., such that for every $i \in \{0, \ldots, n\} $ there is a functor $s_i \colon \Uc_i \rightarrow \Dc$ that is a local right homotopic inverse of $F$, in the sense that $P \circ s_i \simeq \iota_i$ where $\iota_i \colon \Uc_i \rightarrow \CC$ is the inclusion.
\end{defi}

\begin{prop}\label{SG=SC}
    For every bifibration $P\co\E \rightarrow \Ba$ we have that
    $$\sG(P)=\sC(P).$$
\end{prop}

\begin{proof}
    We always have that $\sG(P) \leq \sC(P)$ since a local section is also a homotopic section. To prove the other inequality take any geometric cover $\{\Uc_0, \ldots, \Uc_n \}$ with local homotopic sections $s_i$. For each $i \in \{0,\ldots,n\}$ we have a homotopy $H \colon \Uc_i \times \mathbb{I}_m \rightarrow \Ba$ such that $H_0= P\circ s_i,H_m= \iota_i$ and the following diagram commutes:
    $$
\begin{tikzcd}
\Uc_i \arrow[r, "s_i"] \arrow[d, "i_0"']  & \E \arrow[d, "P"] \\
\Uc_i \times \mathbb{I}_m \arrow[r, "H"'] & \Ba              
\end{tikzcd}
    $$
    Now by the right homotopy lifting property (Proposition \ref{Lifting}) there is a homotopy $\wtilde H  \colon \Uc_i \times \mathbb{I}_m \rightarrow \E$ such that $P \circ \wtilde H _m= H_m= \iota_i$, so $\wtilde H_m$ is a local section of $P$.
\end{proof}

\begin{example}
The bifibration defined in Example \ref{Groupoid}
has no global section since for every functor $S \colon \Ba \rightarrow \E$ we have that $S(h)=\id_{P(\bullet)}$ because there are no other endomorphisms from $P(\bullet)$ into itself. Hence we can not have a section because
$$P\circ S(h)=P(S(h))=P(\id_{P(\bullet)})=\id_\bullet \neq h.$$
In this case the only subcategories of $\Ba=\mathbb{Z}_2$ are the trivial one and the total one, so if we want a geometric cover it must contain the total category. Therefore we can conclude that $\sC(P)=\infty$. 
\end{example}

 We now prove the main result of the paper.

\begin{teo}[\v{S}varc inequality]\label{thm:main_thm}
Let $P \colon \E \rightarrow \Ba$ be a bifibration and let $D$ be a natural system in $\Ba$ with an endopairing. The cup-length in $\ker{P^*}$, denoted by $\cpl(\ker{P^*})$, is a lower bound for the \v{S}varc genus of $P$:  $$\cpl (\ker{P^*}) \leq \sG(P).$$
\end{teo}

\begin{proof}
    We assume that the \v{S}varc genus is finite since otherwise the result is trivial. By Proposition \ref{SG=SC} we know that $\sG(P)=\sC(P)$ so we can assume that we have local sections and not only local homotopic sections. Let $\{\Uc_0,\ldots,\Uc_n\}$ be a geometric cover of $\Ba$ such that every $\Uc_j$ has a local section $s_j$.   Let $\{f_0\},\ldots,\{f_n\}$ be cohomology classes in $\ker P^* \subset \Ho(\Ba;D)$. We  have  $\iota_j=P\circ s_j$  for each $j \in\{0,\ldots,n\}$ and therefore $\iota_j^*=(P\circ s_j)^*=s_j^* \circ P^*$; thus, every class $\{f_j\}$ is in $\ker \iota_j^*$ since it is in $\ker P^*$.

    Using the exact sequence  \eqref{SHORT}:
    $$\Ho^n(\Ba,\Uc_j;D) \xrightarrow{\gamma_j} \Ho^n(\Ba;D) \xrightarrow{\iota_j^*} \Ho^n(\Uc_j;D)$$
we have that every element $\{f_j\}$ is of the form $\{f_j\}=\gamma_j(\{g_j\})$.

    Now if we use the relative cup product defined in Proposition \ref{rel-cup} and using the fact that we have a geometric cover, it follows  that there is a morphism:
    $$\smile\co \Ho(\Ba,\Uc_0;D) \otimes \cdots \otimes \Ho(\Ba,\Uc_n;D) \rightarrow \Ho(\Ba, \Uc_1 \cup \cdots \cup \Uc_n;D).$$
    Moreover we have that $$\Ho(\Ba, \Uc_1 \cup \cdots \cup \Uc_n;D)=\Ho(\Ba,\Ba;D)=0.$$
    
    This entails the wanted result using Lemma \ref{Natural}. Indeed, we know that $$\{f_0\}\smile \ldots \smile \{f_n\}=\{f_0\smile \ldots \smile f_n\}=0.\qedhere$$ 
\end{proof}

We will now examine two examples in which the inequality is applied. In the first, the inequality becomes an equality, whereas in the second, it remains a strict inequality.

\begin{example}
    We know that in the case of Example 
    \ref{Doblecir} the \v{S}varc genus is $1$ as we showed in Example \ref{Svar_ejemplo}. Let us see an example of a natural system in $\Si$ such that $\cpl (\ker P^*)=\sG(P)$. 
    
    Indeed, take the cohomology defined in the example
    of Subsection \ref{subsec:example}. In that case we have that the  coboundary map $\hat{\delta}^0 \colon \tilde{F}^0(\E,P^*D) \rightarrow \tilde{F}^1(\E,P^*D)$ is defined by:
    $$
    \hat{\delta}^0(c_1,c_2,d_1,d_2)=(c_1-d_1,c_2-d_2,[c_1+d_2],[c_2+d_1]),
    $$ where we take an element $f \in F^1(\E ; P^*D)$ to be of the form $$(f(\alpha_1),f(\alpha_2),f(\beta_1),f(\beta_2)).$$
    
    The morphism $P^* \colon \tilde F^1(\Si;D) \rightarrow \tilde F^1(\E ; P^*D)$ is defined by $$P^*(\alpha,\beta)=(\alpha,\alpha,[\beta],[\beta]).$$ So if we take the generator $\{(1,[0])\}$ of $\Ho^1(\Si;D)$ we have that $$P^*(1,[0])=(1,1,[0],[0])=\hat{\delta}^0(0,1,-1,0).$$ Hence we can conclude that $\{(1,[0])\} \in \ker P^*$ and $\mathrm{c.p.l} (P)=1$ for the endopairing defined in Example \ref{Svar_ejemplo}.
\end{example}

\begin{example}\label{ProjectivePlane}
    Let $\mathcal{P}2$ be the category generated by the diagram:
    $$
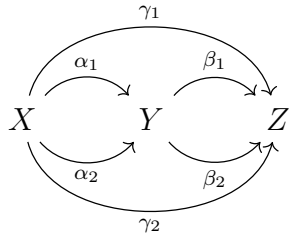
\begin{tikzcd}
X \arrow[r, "\alpha_1", bend left=49] \arrow[r, "\alpha_2"', bend right=49] \arrow[rr, "\gamma_1", bend left=75] \arrow[rr, "\gamma_2"', bend right=75] & Y \arrow[r, "\beta_2"', bend right=49] \arrow[r, "\beta_1", bend left=49] & Z
\end{tikzcd}
    $$
    where $\beta_1 \circ \alpha_1= \gamma_1=\beta_2 \circ \alpha_2$ and $\beta_1 \circ \alpha_2=\gamma_2=\beta_2 \circ \alpha_1$. Let $\E$ be the poset:
     
    \adjustbox{scale=1.3,center}{%
\begin{tikzcd}
X_1 \arrow[d, "\alpha_1^1"'] \arrow[rd, "\alpha_2^1" description, bend left] & X_2 \arrow[ld, "\alpha_1^2" description, bend right] \arrow[d, "\alpha_2^2"] \\
Y_1 \arrow[d, "\beta_1^1"'] \arrow[rd, "\beta_2^1" description, bend left] & Y_2 \arrow[ld, "\beta_1^2" description, bend right] \arrow[d, "\beta_2^2"] \\
Z_1                                                                & Z_2                                                               
\end{tikzcd}
}

    where we denote the morphism between $X_i$ and $Z_j$ by $\gamma_j^i$.
    The functor $P \colon \E \rightarrow \mathcal{P}2$ defined by $P(A_i)=A$ for every object, $P(\alpha_j^i)=\alpha_j$, $P(\beta_j^i)=\beta_i$ and $P(\gamma_j^i)=\gamma_j$ is a covering and therefore a bifibration.

    Let $\mathbb{F}_2$ be the constant functor that takes everything to the finite field with $2$ elements and the pairing given by multiplication. We have that the reduced cochain complex associated to $(\mathcal{P}2,\mathbb{F}_2)$ is equivalent to: $$\mathbb{F}_2^3\rightarrow\mathbb{F}_2^6\rightarrow\mathbb{F}_2^4\rightarrow0\rightarrow\dots$$ where the morphisms are $$\delta^1(x,y,z)=(x+y,x+y,y+z,y+z,x+z,x+z)$$ and $$\delta^2(\alpha_1,\alpha_2,\beta_1,\beta_2,\gamma_1,\gamma_2)=(\alpha_1+\beta_1+\gamma_1,\alpha_2+\beta_2+\gamma_1,\alpha_2+\beta_1+\gamma_2,\alpha_1+\beta_2+\gamma_2).$$

   It follows that $\Ho^1(\mathcal{P}2;\mathbb{F}_2)\cong \mathbb{F}_2$ and $\Ho^2(\mathcal{P}2;\mathbb{F}_2)\cong \mathbb{F}_2$. Moreover, the generator $f$ of $\Ho^1(\mathcal{P}2;\mathbb{F}_2)$ is given by $f=(1,0,0,1,1,0)$ since $\delta^2(\alpha)=0$ and  $f\notin \Ima \delta^1$. Furthermore, 
   \begin{align*}
        f\smile f&=(1,0,0,1,1,0)\smile(1,0,0,1,1,0)\\&=(1\cdot0,0\cdot1,0\cdot1,1\cdot1)\neq0
    \end{align*}
    and since $(0,0,0,1) \notin \Ima \delta^2$ we have that $\{f\} \smile \{f\}\neq 0.$

    Let us compute the cohomology of $\mathcal{E}$ with coefficients in $\mathbb{F}_2$. We have that the reduced cochain complex is equivalent to: $$\mathbb{F}_2^6\rightarrow\mathbb{F}_2^{12}\rightarrow\mathbb{F}_2^8\rightarrow0\rightarrow\dots$$ where the morphisms are \begin{align*}
    \delta^1(x_1,x_2,y_1,y_2,z_1,z_2)=\\(x_1+y_1,x_1+y_2,x_2+y_1,x_2+y_2,y_1+z_1,y_1+z_2,\\y_2+z_1,y_2+z_2,x_1+z_1,x_1+z_2,x_2+z_1,x_2+z_2)
    \end{align*} and \begin{align*}\delta^2(\alpha_1^1,\alpha_2^1,\alpha^2_1,\alpha^2_2,\beta_1^1,\beta^1_2,\beta_1^2,\beta_2^2,\gamma_1^1,\gamma_2^1,\gamma_1^2,\gamma_2^2)=\\(\alpha_1^1+\beta_1^1+\gamma_1^1,\alpha_1^1+\beta_2^1+\gamma^1_2,\alpha_2^1+\beta_1^2+\gamma^1_1,\alpha_2^1+\beta_2^2+\gamma^1_2,\\
    \alpha_1^2+\beta_1^1+\gamma^2_1,\alpha_1^2+\beta_2^1+\gamma^2_2,\alpha_2^2+\beta_1^2+\gamma^2_1,\alpha_2^2+\beta_2^2+\gamma^2_2)
    \end{align*}
One can check that $\Ho^1(\mathcal{E};\mathbb{F}_2)=0$ so the generator $\alpha$ of $\Ho^1(\mathcal{P}2;\mathbb{F}_2)$ is in $\ker P^*$ and $\mathrm{c.p.l}(P)=2$ because in this case $\{\alpha\} \smile \{\alpha\} \neq 0.$ 
    
    Using the inequality of Theorem \ref{thm:main_thm} we know that we need at least $3$ subcategories in order to cover $\mathcal{P}2$ with a geometric cover with the property of having a local section. In this case, the inequality is strict since one can check that we need $4$. Indeed, suppose that we had only three subcategories. Since the cover is geometric, by the pigeonhole principle there is at least one subcategory that has two $2$-chains, i.e., it contains one of the following subcategories:
    $$
\begin{tikzcd}
X \arrow[r, "\alpha_i"] & Y \arrow[r, "\beta_2"', bend right=49] \arrow[r, "\beta_1", bend left=49] & Z & X \arrow[r, "\alpha_1", bend left=49] \arrow[r, "\alpha_2"', bend right=49] & Y \arrow[r, "\beta_i"] & Z
\end{tikzcd}
    $$
    In both cases we can prove that it has no local section. We will see this for the case on the right since the left one is analogous. Let $S \colon \Uc \rightarrow \E $ be a local section. The image of the object $X$ would be $S(X)=X_i$. In turn, $\alpha_1$ and $\alpha_2$ would have to be lifted to maps $S(\alpha_1) \colon X_i \rightarrow S(Y)$ and $S(\alpha_i) \colon X_i \rightarrow S(Y)$. Since $P \circ S(\alpha_i)=\alpha_i$ then $S(\alpha_1)\neq S(\alpha_2)$ and they terminate in different objects but this is a contradiction since $S(Y)$ is the same object.
\end{example}

\bibliographystyle{plain}
\bibliography{biblio_svarc}



\end{document}